\documentclass[a4paper,11pt]{amsart}
\usepackage[latin1]{inputenc}
\usepackage{amsthm}
\usepackage{latexsym,amsmath,amsfonts,amscd,amssymb}
\usepackage{graphicx}
\usepackage{xypic}
\usepackage{array,booktabs,arydshln}
\usepackage{setspace}
\usepackage{fancyhdr}
\usepackage{fancybox}
\usepackage{multicol}
\usepackage{paralist}
\usepackage[all]{xy}\SelectTips{cm}{}
\usepackage{verbatim}
\usepackage{color}

\newtheorem{theorem}{Theorem}[section]
\newtheorem{lemma}[theorem]{Lemma}
\newtheorem{proposition}[theorem]{Proposition}
\newtheorem{corollary}[theorem]{Corollary}
\newtheorem{definition}[theorem]{Definition}

\newtheorem{notation}[theorem]{Notation}

\theoremstyle{definition}

\newtheorem{Remark}[theorem]{Remark}

\newtheorem{Example}[theorem]{Example}

\newcommand{\opn}{{\mathcal{O}_{\mathbb{P}^n}}}
\newcommand{\pn}{{\mathbb{P}^n}}
\newcommand{\CC}{{\mathbb C}}

\newcommand{\ZZ}{{\mathbb Z}}
\newcommand{\PP}{{\mathbb P}}
\newcommand{\GG}{{\mathbb G}}

\newcommand{\A}{{\mathcal A}}

\newcommand{\OO}{{\mathcal O}}

\newcommand{\Pic}{\operatorname{Pic}}

\newcommand{\rk}{\operatorname{rk}}
\newcommand{\Hom}{\operatorname{Hom}}

\newcommand{\im}{\operatorname{im}}
\newcommand{\codim}{\operatorname{codim}}
\newcommand{\coker}{\operatorname{coker}}

\newcommand{\Fitt}{\operatorname{Fitt}}
\newcommand{\Ann}{\operatorname{Ann}}
\newcommand{\h}[1]{\-\mbox{-#1}}
\makeindex

\newlength\linhaxy
\newlength\linhaxylabels
\newlength\colunaxy
\newlength\colunaxylabels

\newcommand{\morphismtopn}{\varphi}

\title{Monads on projective varieties}
\author{Simone Marchesi}
\address{Instituto de Matem\'{a}tica, Estat\'{\i}stica e Computa\c{c}\~{a}o Cient\'{\i}fica, Universidade Estadual de Campinas, Rua S\'{e}rgio Buarque de Holanda, 651, Cidade Universit\'{a}ria ``Zeferino Vaz'', Distrito Bar\~{a}o Geraldo, Campinas, CEP 13083--859, S\~{a}o Paulo, Brasil}
\email{marchesi@ime.unicamp.br}

\author{Pedro Macias Marques}
\address{Departamento de Matem\'{a}tica, Escola de Ci\^{e}ncias e Tecnologia, Centro de Investiga\c{c}\~{a}o em Mate\-m\'{a}\-tica e Aplica\c{c}\~{o}es, Instituto de Investiga\c{c}\~{a}o e Forma\c{c}\~{a}o Avan\c{c}ada, Universidade de \'{E}vora, Rua Rom\~{a}o Ramalho, 59, P--7000--671 \'{E}vora, Portugal}
\email{pmm@uevora.pt}

\author{Helena Soares}
\address{Departamento de Matem\'{a}tica, Instituto Universit\'{a}rio de Lisboa (ISCTE - IUL), UNIDE (BRU - Business Research Unit), Av.\ das For\c{c}as Armadas, 1649--026 Lisboa, Portugal}
\email{helena.soares@iscte.pt}

\date{}



\begin{document}
\maketitle

\begin{abstract}
We generalise Fl\o{}ystad's theorem on the existence of monads on the projective space to a larger set of projective varieties. We consider a variety $X$, a line bundle $L$ on $X$, and a base\h{point}\h{free} linear system of sections of $L$ giving a morphism to the projective space whose image is either arithmetically Cohen\h{Macaulay} (ACM), or linearly normal and not contained in a quadric. We give necessary and sufficient conditions on integers $a$, $b$, and $c$ for a monad of type
\begin{equation*}
\xymatrix@R=\linhaxy@C=\colunaxy{0\ar[r]&(L^\vee)^a\ar[r]&\OO_{X}^{\,b}\ar[r]&L^c\ar[r]&0}
\end{equation*}
to exist. We show that under certain conditions there exists a monad whose cohomology sheaf is simple. We furthermore characterise low\h{rank} vector bundles that are the cohomology sheaf of some monad as above.

Finally, we obtain an irreducible family of monads over the projective space and make a description on how the same method could be used on an ACM smooth projective variety $X$. We establish the existence of a coarse moduli space of low\h{rank} vector bundles over an odd\h{dimensional} $X$ and show that in one case this moduli space is irreducible.
\end{abstract}

\section{Introduction}

A \emph{monad} over a projective variety $X$ is a complex
\[
M_\bullet\colon
\xymatrix@R=\linhaxylabels@C=\colunaxylabels{0\ar[r] &A\ar[r]^f &B\ar[r]^g &C\ar[r] &0}
\]
of morphisms of coherent sheaves on $X$, where $f$ is injective and $g$ is surjective. The coherent sheaf ${E:=\ker g / \im f}$ is called the \textit{cohomology sheaf} of the monad $M_\bullet$. This is one of the simplest ways of constructing sheaves, after kernels and cokernels.

The first problem we need to tackle when studying monads is their existence. Fl\o{}ystad gave sufficient and necessary conditions for the existence of monads over the projective space whose maps are given by linear forms, in \cite{Flo00}. In \cite{CMR09}, Costa and Miró\h{Roig} extended this result to smooth quadric hypersurfaces of dimension at least three, and in \cite{Jar07}, Jardim made a further generalisation to any hypersurface in the projective space. We can find additional partial results on the existence of monads in the literature, by means of construction of examples of monads over other projective varieties (for instance, blow\h{ups} of the projective plane in \cite{Buc04}, Abelian varieties in \cite{Gul13}, Fano threefolds in \cite{Fae14} and \cite{Kuz12}, complete intersection, Calabi\h{Yau} threefolds in \cite{HJ13}, and Segre varieties in \cite{MMS14}). In \cite{JMR08} the authors expressed the wish of having a generalisation of the results on existence by Fl\o{}ystad and by Costa and Miró\h{Roig} to varieties other than the projective space and quadric hypersurfaces. Here we generalise Fl\o{}ystad's theorem to a larger set of projective varieties. We let $X$ be a variety of dimension $n$ and $L$ be a line bundle on $X$. We consider a linear system ${V\subseteq H^0(L)}$, without base points, defining a morphism ${\morphismtopn:X\to\PP(V)}$  and suppose that its image ${X'\subset\PP(V)}$ is arithmetically Cohen\h{Macaulay} (ACM) (see Definition~\ref{ACMdefn} and Theorem~\ref{maintheorem}) or linearly normal and not contained in a quadric hypersurface (Theorem~\ref{notcutbyquadricstheorem}). Then we give necessary and sufficient conditions on integers $a$, $b$, and $c$ for a monad of type
\begin{equation}\label{genmonad0}\tag{M}
\xymatrix@R=\linhaxy@C=\colunaxy{0\ar[r]&(L^\vee)^a\ar[r]&\OO_{X}^{\,b}\ar[r]&L^c\ar[r]&0}
\end{equation}
to exist.

Once existence of monads over a variety $X$ is proved, we can study their cohomology sheaf. One of the most interesting questions to ask is whether this sheaf is stable and this has been established in special cases (see \cite{AO94} and \cite{JMR08} for instance). Since stable sheaves are simple, a common approach is to study simplicity (in \cite{CMR09} the authors show that any mathematical instanton bundle over an odd\h{dimensional} quadric hypersurface is simple, and in particular that it is stable over a quadric threefold). We show that under certain conditions, in the case when $X'$ is ACM, there exists a monad of type \eqref{genmonad0} whose cohomology sheaf is simple (Proposition~\ref{simpleprop}).

As we said, monads are a rather simple way of obtaining new sheaves. When the sheaf we get is locally\h{free}, we may consider its associated vector bundle, and by abuse of language we will not distinguish between one and the other. There is a lot of interest in low\h{rank} vector bundles over a projective variety $X$, i.e.\ those bundles whose rank is lower than the dimension of $X$, mainly because they are very hard to find. We characterise low\h{rank} vector bundles that are the cohomology sheaf of a monad of type \eqref{genmonad0} (Theorem~\ref{lowerrankth}).

Finally, we would like to be able to describe families of monads, or of sheaves coming from monads. There has been much work done on this since the nineties. Among the properties studied on these families is irreducibility (see for instance \cite{Tik12} and \cite{Tik13} for the case of instanton bundles over the projective space). Here we obtain an irreducible family of monads over the projective space (Theorem \ref{theorem2}),  and make a description on how the same method could be used on another ACM projective variety. Furthermore, we establish the existence of a coarse moduli space of low\h{rank} vector bundles over an odd\h{dimensional}, ACM projective variety (Theorem \ref{modulispace}), and show that in one case this moduli space is irreducible (Corollary \ref{irreduciblemodulispace}).

\bigskip

\noindent\textbf{Acknowledgements.}
The authors wish to thank Enrique Arrondo, Laura Costa, Marcos Jardim, Rosa María Miró\h{Roig}, and Daniela Prata for fruitful discussions. We would also like to thank the referee for the useful comments and observations.

We would like to thank the Universidade Estadual de Campinas (IMECC-UNICAMP) for the hospitality and for providing the best working conditions. The second author also wishes to thank Northeastern University and KU Leuven for their hospitality.

The first author was partially supported by Funda\c{c}\~{a}o de Amparo \`{a} Pesquisa do Estado de S\~{a}o Paulo (FAPESP), grant 2017/03487-9. The second and third authors were partially supported by Funda\c{c}\~{a}o para a Ci\^{e}ncia e Tecnologia (FCT), project ``Comunidade Portuguesa de Geometria Alg\'{e}brica'', PTDC/MAT-GEO/0675/2012. The second author was also partially supported by FCT sabbatical leave grant SFRH/BSAB/1392/2013, by CIMA -- Centro de Investiga\c{c}\~{a}o em Matem\'{a}tica e Aplica\c{c}\~{o}es, Universidade de \'{E}vora, project PEst-OE/MAT/UI0117/2014, and by FAPESP Visiting Researcher Grant 2014/12558--9. The third author was also partially supported by FCT sabbatical leave grant SFRH/BSAB/105740/2014, by BRU - Business Research Unit, Instituto Universit\'{a}rio de Lisboa (ISCTE - IUL), and by FAPESP Visiting Researcher Grant 2014/00498-1.

\section{Monads over ACM varieties}

Let $X$ be a projective variety of dimension $n$ over an algebraically closed field $k$, let $L$ be a line bundle on $X$, and let ${V\subseteq H^0(L)}$ yield a linear system without base points, defining a morphism ${\morphismtopn:X\to\PP(V)}$. Our main goal is to study monads over $X$ of type
\[
\xymatrix@R=\linhaxy@C=\colunaxy{0\ar[r] &(L^\lor)^a\ar[r]&\OO_X^{\,b} \ar[r] &L^c\ar[r] &0.}
\]
In this section we recall the concept of monad, as well as the results that were the starting point for the present paper, i.e.\ Fl{\o}ystad's work regarding the existence of monads on the projective space (\cite{Flo00}).

Let us first fix the notation used throughout the paper.

\begin{notation}
Let $Y\subseteq\PP^N$ be a projective variety of dimension $n$ over an algebraically closed field $k$. Let $R_Y$ be the homogeneous graded coordinate ring of $Y$ and $\mathcal{I}_{Y/\PP^N}$ its ideal sheaf.

If $\mathcal{E}$ is a coherent sheaf over $Y$ we will denote its dual by $\mathcal{E}^\lor$. We also denote the graded module ${H^i_*(Y,\mathcal{E})=\oplus_{m\in\ZZ} H^i\big(Y,\mathcal{E}(m)\big)}$ and ${h^i(\mathcal{E})=\dim H^i(Y,\mathcal{E})}$.

Given any $k$-vector space $V$, we will write $V^*$ to refer to its dual.
\end{notation}

\begin{definition}\label{ACMdefn}
Let $Y$ be a projective variety of dimension $n$ over an algebraically closed field $k$. We say that $Y$ is arithmetically Cohen\h{Macaulay} (ACM) if its graded coordinate ring $R_Y$ is a Cohen\h{Macaulay} ring.
\end{definition}

\begin{Remark}
If ${Y\subseteq\PP^N}$ is a projective variety then being ACM is equivalent to the following vanishing:
\[
H^1_*\big(\PP^N,\mathcal{I}_{Y/\PP^N}\big)=0, \quad H^i_*(Y,\OO_Y)=0, \quad 0<i<n.
\]
Moreover, we note that the notion of ACM variety depends on the embedding.
\end{Remark}

The first problem we will address concerns the existence of monads on projective varieties (see Section \ref{SectionExistence}) and the generalisation of the following result.

\begin{theorem}[\cite{Flo00}, Main Theorem and Corollary 1]\label{TeoFlo}
Let $N\geq 1$. There exists a monad of type
\begin{equation}\label{monadFlo}
\xymatrix@R=\linhaxylabels@C=\colunaxylabels{0\ar[r]&\OO_{\PP^N}(-1)^a\ar[r]^-{f}&\OO_{\PP^N}^{\,b}\ar[r]^-{g}&\OO_{\PP^N}(1)^c\ar[r]&0}
\end{equation}
if and only if one of the following conditions holds:
\begin{enumerate}[(i)]
\item  $b \geq a + c$ and $b \geq 2c+N-1$, \label{monadFlo1}
\item  $b \geq a+c+N$. \label{monadFlo2}
\end{enumerate}
If so, there actually exists a monad with the map $f$ degenerating in expected codimension $b-a-c+1$.

If the cohomology of the monad \eqref{monadFlo} is a vector bundle of rank $<N$ then $N=2l+1$ is odd and the monad has the form
\begin{equation}\label{monadInst}
\xymatrix@R=\linhaxy@C=\colunaxy{0\ar[r]&\OO_{\PP^{2l+1}}(-1)^c\ar[r]&\OO_{\PP^{2l+1}}^{\,2l+2c} \ar[r]&\OO_{\PP^{2l+1}}(1)^c\ar[r]&0.}
\end{equation}
Conversely, for every ${c, l\geq 0}$ there exist monads of type \eqref{monadInst} whose cohomology is a vector bundle.
\end{theorem}

Observe that the vector bundles which are the cohomology of a monad of the form \eqref{monadInst} are the so\h{called} instanton bundles.

The next construction uses standard techniques of projective geometry and it explains somehow the way we thought Fl\o{}ystad's case could be generalised to other projective varieties.

\bigskip

Let $X'$ be the image of $X$ in $\PP(V)$. Taking ${N=\dim V-1}$, let $\PP^N:=\PP(V)$ and  ${m := \codim_{\PP^N} X'}$. Consider a monad of type \eqref{monadFlo} and take a projective linear subspace ${\Lambda \subset \PP^N}$ of dimension ${m-1}$ such that ${\Lambda \cap X' = \emptyset}$. Fixing coordinates ${z_0,\ldots, z_N}$ in $\PP^N$ we may assume that  ${I(\Lambda)=(z_{0},\ldots,z_{N-m})}$.

Let $A$ and $B$ be the matrices associated to the morphisms $f$ and $g$, respectively, in \eqref{monadFlo}. Consider the induced morphisms $\tilde{f}$ and $\tilde{g}$ whose matrices are, respectively, $\tilde{A}$ and $\tilde{B}$, obtained from $A$ and $B$ by the vanishing of the linear forms that define a linear complement of $\Lambda$, i.e.\ ${\tilde{f}= f_{|\{z_{N-m+1} =\cdots =z_N = 0\}}}$ and ${\tilde{g}= g_{|\{\{z_{N-m+1} =\cdots =z_N = 0\}}}$.

By construction we have $\tilde{B}\tilde{A} = 0$. If $x \in \Lambda$ then the ranks of $\tilde{A}$ and $\tilde{B}$ evaluated at $x$ are no longer maximal, that is, $\rk(\tilde{A})<a$ and $\rk(\tilde{B})<c$. In particular, the complex
\[
\xymatrix@R=\linhaxylabels@C=\colunaxylabels{\OO_{\PP^N}(-1)^a\ar[r]^-{\tilde{f}}&\OO_{\PP^N}^{\,b}\ar[r]^-{\tilde{g}}&\OO_{\PP^N}(1)^c}
\]
is not a monad on $\PP^N$ anymore.
Nevertheless, for a general $x \in X'$, the matrices $\tilde{A}(x)$ and $\tilde{B}(x)$ have maximal rank and hence the complex
\[
\xymatrix@R=\linhaxylabels@C=\colunaxylabels{0\ar[r] &(L^\lor)^a\ar[r]^-{\morphismtopn^*\tilde{f}}&\OO_{X}^{\,b}\ar[r]^{\morphismtopn^*\tilde{g}}&L^c\ar[r] & 0,}
\]
where ${L=\morphismtopn^*\big(\OO_{\PP^N}(1)\big)}$,  is a monad on $X$.

\section{Existence of monads over ACM varieties}\label{SectionExistence}

The aim of this section is to prove two characterisations of the existence of monads on projective varieties. We start by giving sufficient conditions for a monad to exist.

\begin{lemma}\label{existencelemma}
Let $X$ be a variety of dimension $n$, let $L$ be a line bundle on $X$, and let ${V\subseteq H^0(L)}$ be a linear system, with no base points, defining a morphism ${X\to\PP(V)}$. Suppose $a$, $b$, and $c$ are integers such that one of the following conditions holds:
\begin{enumerate}[(i)]
\item  $b \geq a + c$ and $b \geq 2c+n-1$, \label{genmonad1}
\item  $b \geq a+c+n$. \label{genmonad2}
\end{enumerate}
Then there exists a monad of type
\begin{equation}\label{genmonad}\tag{M}
\xymatrix@R=\linhaxylabels@C=\colunaxylabels{0\ar[r]&(L^\lor)^a\ar[r]^-{f}&\OO_{X}^{\,b}\ar[r]^-{g}&L^c\ar[r]&0}
\end{equation}
Moreover, the map $f$ degenerates in expected codimension $b-a-c+1$ and $g$ can be defined by a matrix whose entries are global sections of $L$ that span a subspace of $V$ whose dimension  is ${\min\big(b-2c+2,\dim V\big)}$.
\end{lemma}
The main ideas of the proof follow Floystad's construction, combined with the projective geometry standard results described at the end of the last section. Observe that, under the hypotheses of the theorem, the existence of a monad \eqref{genmonad} is equivalent to the existence of a monad
\[
\xymatrix@R=\linhaxylabels@C=\colunaxylabels{0\ar[r]&\OO_{X'}(-1)^a\ar[r]^-{f}&\OO_{X'}^{\,b}\ar[r]^-{g}&\OO_{X'}(1)^c\ar[r]&0.}
\]
\begin{proof}
Let ${N=\dim V-1}$ and write $\PP^N$ for $\PP(V)$. Suppose that one of the conditions \eqref{genmonad1} and \eqref{genmonad2} holds. If $b$ is high enough with respect to $a$ and $c$ so that ${b\geq 2c+N-1}$ or ${b \geq a+c+N}$, then by Theorem \ref{TeoFlo}, there is a monad
\[
\xymatrix@R=\linhaxylabels@C=\colunaxylabels{0\ar[r]&\OO_{\PP^N}(-1)^a\ar[r]^-{f}&\OO_{\PP^N}^{\,b}\ar[r]^-{g}&\OO_{\PP^N}(1)^c\ar[r]&0.}
\]
By restricting morphisms $f$ and $g$ to $X'$, we get a monad of type \eqref{genmonad}. So from here on we may assume that  ${b < a+c+N}$ and ${b < 2c+N-1}$.

Suppose first that condition \eqref{genmonad1} in Theorem \ref{maintheorem} is satisfied. Then ${N-1 > b-2c \ge n-1}$, so
\[
0\le N-(b-2c+2)\le N-n-1.
\]
Therefore we can take a projective linear subspace ${\Lambda \subset \PP^N}$ of dimension ${N-(b-2c+2)}$, disjoint from $X'$, and choose linearly independent sections ${z_{0},\ldots,z_{b-2c+1}\in H^0\big(\OO_{X'}(1)\big)}$ such that ${I(\Lambda)=(z_{0},\ldots,z_{b-2c+1})}$. Let us divide the coordinate set $\{z_0,\ldots,z_{b-2c+1}\}$ into two subsets:\linebreak ${x_0,\ldots,x_p}$ and ${y_0,\ldots,y_q}$, with ${\lvert p-q\rvert \leq 1}$ and such that ${b-2c=p+q}$. Define the matrices
\[
X_{c,c+p} = 
\begin{bmatrix}
x_0 & x_1 & \cdots & x_p\\
& x_0 & x_1 & \cdots & x_p\\
& & \ddots & & & \ddots\\
& & & x_0 & x_1 & \cdots & x_p
\end{bmatrix}
\]
and
\[
Y_{c,c+q} =
\begin{bmatrix}
y_0 & y_1 & \cdots & y_q\\
& y_0 & y_1 & \cdots & y_q\\
& & \ddots & & & \ddots\\
& & & y_0 & y_1 & \cdots & y_q
\end{bmatrix},
\]
of sizes $c\times(c+p)$ and $c\times(c+q)$, respectively. Therefore, the matrices
\[
B = 
\begin{bmatrix}
X_{c,c+p} & Y_{c,c+q}
\end{bmatrix}, 
\quad 
A = 
\begin{bmatrix}
Y_{c+p,c+p+q} \\
 - X_{c+q,c+q+p}
\end{bmatrix}
\]
allow us to construct the following complex on $X'$:
\[
\xymatrix@R=\linhaxylabels@C=\colunaxylabels{\OO_{X'}(-1)^{c+p+q} \ar[r]^-{f}_-{A} & \OO_{X'}^{\,2c+p+q} \ar[r]^-{g}_-{B} & \OO_{X'}(1)^c \ar[r] & 0.
}
\]
By construction $BA=0$ and $\rk B(x) = c$, for each $x\in X'$.

Our next goal is to construct an injective morphism on $X'$,
\begin{equation}\label{phiX}
\xymatrix@R=\linhaxylabels@C=\colunaxylabels{\OO_{X'}(-1)^{c+p+q-s} \ar[r]^-{\phi} & \OO_{X'}(-1)^{c+p+q}}
\end{equation}
so that we are able to compute the expected codimension of the degeneracy locus of the composition $f \circ \phi$, i.e.\ the codimension of
\[
Z_s = \{ x \in X' \mid \rk (f \circ \phi) (x) < c+p+q-s\}.
\]
Observe that the matrices $A$ and $B$ define two more complexes: one complex on an $n$-dimensional projective subspace $\PP^n\subset \PP^N$,
\[
\xymatrix@R=\linhaxylabels@C=\colunaxylabels{
\OO_{\PP^n}(-1)^{c+p+q} \ar[r]^-{\hat{f}}_-{A} & \OO_{\PP^n}^{\,2c+p+q} \ar[r]^-{\hat{g}}_-{B} & \OO_{\PP^n}(1)^c \ar[r] & 0,
}
\]
such that ${\PP^n \cap \Lambda = \emptyset}$, and another one on $\PP^N$ given by
\[
\xymatrix@R=\linhaxylabels@C=\colunaxylabels{
\OO_{\PP^N}(-1)^{c+p+q} \ar[r]^-{\bar{f}}_-{A} & \OO_{\PP^N}^{\,2c+p+q} \ar[r]^-{\bar{g}}_-{B} & \OO_{\PP^N}(1)^c.
}
\]
Consider a generic injective morphism
\[
\xymatrix@R=\linhaxylabels@C=\colunaxylabels{\OO_{\PP^n}(-1)^{c+p+q-s} \ar[r]^-{\hat{\phi}} &  \OO_{\PP^n}(-1)^{c+p+q},}
\]
$s\geq 0$, inducing both a morphism $\OO_{\PP^N}(-1)^{c+p+q-s} \stackrel{\bar{\phi}}{\longrightarrow} {\OO_{\PP^N}}(-1)^{c+p+q}$ and a morphism $\phi$ as in \eqref{phiX}. Note that the three morphisms are represented by the same matrix.

From Lemmas 2 and  3 in \cite{Flo00} it follows that the expected codimension of the degeneracy locus $\hat{Z}_s$ of $\hat{f}\circ\hat{\phi}$ is at least $ s+1$. Moreover, denoting the degeneracy locus of $\bar{f}\circ\bar{\phi}$ by $\bar{Z}_s$, we have the following relations:
\[
\bar{Z}_s =\textstyle \bigcup_{x \in \hat{Z}_s} \langle x,\Lambda \rangle, \quad Z_s = \bar{Z}_s \cap X'.
\]
Observe that the fact that $\hat{\phi}$ is injective implies that $\phi$ is also injective.
Computing dimensions, we obtain that $\codim_{\PP^N}\bar{Z}_s\geq s+1$ and thus
\[
\codim_{X'} Z_s \geq s+1.
\]
Then, taking ${s=c+p+q-a=b-a-c\geq 0}$, the complex
\[
\xymatrix@R=\linhaxy@C=\colunaxy{0\ar[r]&\OO_{X'}(-1)^a\ar[r]&\OO_{X'}^{\,b}\ar[r]&\OO_{X'}(1)^c\ar[r]&0}
\]
is a monad on $X'$ since we have $\codim_{X'} Z_s \geq s+1= b-c-a+1$ (so, $\dim Z_s\leq n-1)$.

\medskip

Now, suppose condition (ii) holds, i.e.\ ${b\geq a+c+n}$, and suppose that ${b<2c+n-1}$ (otherwise we would be again in case (i)). Hence, $c> a+1$ and $b>2a+n+1>2a+n-1$.

Applying case (i) to the inequalities ${b\geq a+c+n>a+c}$ and ${b>2a+n-1}$, we know there is a monad on $X'$ of type
\[
\xymatrix@R=\linhaxy@C=\colunaxy{0\ar[r]&\OO_{X'}(-1)^c\ar[r]&\OO_{X'}^{\,b}\ar[r]&\OO_{X'}(1)^a\ar[r]&0}
\]
where the map $\OO_{X'}(-1)^c \longrightarrow \OO_{X'}^{\,b}$ degenerates in codimension at least $b-a-c+1\geq n+1$. Dualising this complex, we get
\[
\xymatrix@R=\linhaxy@C=\colunaxy{0\ar[r]&\OO_{X'}(-1)^a\ar[r]&\OO_{X'}^{\,b}\ar[r]&\OO_{X'}(1)^c\ar[r]&0}
\]
which is still a monad on $X'$, for the codimension of the degeneracy locus of $\OO_{X'}(-1)^a \longrightarrow \OO_{X'}^{\,b}$ is at least $b-a-c+1$.
\end{proof}

\medskip

\begin{Remark}\label{rmk-pullback}
We could have constructed a monad on $X$ just by taking the pullback of a monad on $\PP^n$ and applying Floystad's result. In fact, we could have defined a finite morphism $X \rightarrow \PP^n$ by considering precisely $\dim X +1$ linearly independent global sections of $L$ (and not vanishing simultaneously at any point $x\in X$). The pullback via this morphism of a monad on $\PP^n$ would give us a monad on $X$.
Nevertheless, we note that the construction above is far more general. It allows us to use a bigger number of global sections and it also provides an explicit construction of the monad on $X$.
\end{Remark}

We next prove the two main results of this section, which generalise Fl\o{}ystad's theorem on the existence of monads on the projective space. We consider a variety $X$, a line bundle $L$ on $X$, and a base\h{point}\h{free} linear system of sections of $L$ giving a morphism to the projective space. Each result asks different properties on the image ${X' \subset \PP(V)}$ of the variety $X$.

Our first result characterises the existence of monads of type \eqref{genmonad} in the case when $X'$ is an ACM projective variety.

\begin{theorem}\label{maintheorem}
Let $X$ be a variety of dimension $n$ and let $L$ be a line bundle on $X$. Suppose there is a linear system ${V\subseteq H^0(L)}$, with no base points, defining a morphism ${X\to\PP(V)}$ whose image ${X'\subset\PP(V)}$ is a projective ACM variety. Then there exists a monad of type
\begin{equation}\label{genmonad}\tag{M}
\xymatrix@R=\linhaxylabels@C=\colunaxylabels{0\ar[r]&(L^\lor)^a\ar[r]^-{f}&\OO_{X}^{\,b}\ar[r]^-{g}&L^c\ar[r]&0}
\end{equation}
if and only if one of the following conditions holds:
\begin{enumerate}[(i)]
\item  $b \geq a + c$ and $b \geq 2c+n-1$, \label{genmonad1}
\item  $b \geq a+c+n$. \label{genmonad2}
\end{enumerate}
If so, there actually exists a monad with the map $f$ degenerating in expected codimension $b-a-c+1$. Furthermore, $g$ can be defined by a matrix whose entries are global sections of $L$ that span a subspace of $V$ whose dimension  is ${\min\big(b-2c+2,\dim V\big)}$.
\end{theorem}

Note that if condition (ii) in the above theorem is satisfied then there exists a monad whose cohomology is a vector bundle of rank greater than or equal to the dimension of $X$.

\bigskip

\begin{proof}

The existence of the monad in case conditions (i) or (ii) are satisfied follows from Lemma \ref{existencelemma}. Let us show that these conditions are necessary.
Suppose we have a monad on $X'$
\[
\xymatrix@R=\linhaxylabels@C=\colunaxylabels{0\ar[r]&\OO_{X'}(-1)^a\ar[r]^-{f}&\OO_{X'}^{\,b}\ar[r]^-{g}&\OO_{X'}(1)^c\ar[r]&0.}
\]
This immediately implies that $b\geq a+c$.
The image of the induced map $H^0(\OO_{X'}^{\,b})\to H^0\big(\OO_{X'}(1)^c\big)$ defines a vector subspace $U'\subset H^0\big(\OO_{X'}(1)^c\big)$ which globally generates $\OO_{X'}(1)^c$. In particular, there is a diagram
\[
\xymatrix@R=\linhaxylabels@C=\colunaxylabels{
\OO_{X'}^{\,b} \ar[d]\ar[r]^-{g} & \OO_{X'}(1)^c \ar[r] & 0\\
U'\otimes \OO_{X'} \ar[d]\ar[ur]_-{\tilde{g}} \\
0
}
\]
Since $\OO_{X'}(1)^c$ is globally generated via $\tilde{g}$, we have ${\dim U'\geq c+n}$, otherwise the degeneracy locus of $\tilde{g}$ would be non\h{empty}.

Let $U\subset U'$ be a general subspace with $\dim U=c+n-1$. Hence the map $\tilde{p}:U\otimes \OO_{X'}\to \OO_{X'}(1)^c$, induced by $\tilde{g}$, degenerates in dimension zero. Take a splitting
\[
\xymatrix@R=\linhaxy@C=\colunaxy{H^0(\OO_{X'}^{\,b}) \ar@<.5ex>[r] & U' \ar@<.5ex>[l]}
\]
and define $W=H^0(\OO_{X'}^{\,b})/U$. Denote $\mathcal{I}=\mathcal{I}(X')\subset k[z_0,\ldots,z_N]$. Let $S=k[z_0,\ldots,z_N]/\mathcal{I}$ be the coordinate ring of $X'$. Since $X'$ is projectively normal, $S$ is integrally closed and therefore ${S=H^0_*(\OO_{X'})}$, so we have the following commutative diagram of graded $S$-modules:
\[
\xymatrix@R=\linhaxylabels@C=\colunaxylabels{
& U\otimes S \ar[d] \ar@{=}[r] & U\otimes S \ar[d]^-{p} \\
S(-1)^a \ar[r] \ar@{=}[d] & S^b \ar[r]\ar[d] & S(1)^c\\
S(-1)^a \ar[r]_-{q}  & W\otimes S}
\]
Sheafifying the above diagram, we get a surjective map
\[
\xymatrix@R=\linhaxy@C=\colunaxy{\coker\tilde{q}\ar[r] & \coker\tilde{p} \ar[r] & 0.}
\]
Because $\tilde{p}$ degenerates in the expected codimension we have, by Theorem 2.3 in \cite{BE77b},
\[
\Fitt_1(\coker\tilde{p})=\Ann(\coker\tilde{p}),
\]
and so we obtain the following chain of inclusions
\[
\Fitt_1(\coker \tilde{q})\subset \Ann(\coker\tilde{q})\subset \Ann(\coker\tilde{p})=\Fitt_1(\coker \tilde{p}),
\]
where the first inclusion follows from Proposition 20.7.a in \cite{Eis94}. Therefore,
\[
\Fitt_1(\coker q)\subset H^0_{*}\big(\Fitt_1(\coker\tilde{q})\big)\subset H^0_{*}\big(\Fitt_1(\coker\tilde{p})\big).
\]
Since $p$ degenerates in expected codimension $n$ and $X'$ is ACM, $S/{\Fitt_1(\coker p)}$ is a Cohen\h{Macaulay} ring of dimension $1$ (see \cite{Eis94}, Theorem 18.18). In particular, $\Fitt_1(\coker p)$ is a saturated ideal because the irrelevant maximal ideal $\mathfrak{m}\subset S$ is not an associate prime of it, and thus
\[
H^0_*\big(\Fitt_1(\coker\tilde{p})\big)=\Fitt_1(\coker p).
\]
By definition, $\Fitt_1(\coker p)$ is generated by polynomials of degree at least $c$, so all polynomials in $\Fitt_1(\coker q)$ must also have degree at least $c$. Note that the map $q$ may be assumed to have generic maximal rank for $f$ is injective and $S^b\to W\otimes S$ is a general quotient. This leads to two possibilities: either $\dim W\geq c$ or $\dim W> a$. Recalling that $\dim W=b-c-n+1$, we obtain respectively
\[
b\geq 2c+n-1
\]
or
\[
b\geq a+c+n,
\]
and this concludes the proof.
\end{proof}

\bigskip

We now state the second characterisation result, with a similar setting as in Theorem \ref{maintheorem}, except that we drop the hypothesis that the image $X'$ of $X$ in $\PP(V)$ is ACM, and assume instead that it is linearly normal and not contained in a quadric.

\begin{theorem}\label{notcutbyquadricstheorem}
Let $X$ be a variety of dimension $n$ and let $L$ be a line bundle on $X$. Suppose there is a linear system ${V\subseteq H^0(L)}$, with no base points, defining a morphism ${X\to\PP(V)}$ whose image ${X'\subset\PP(V)}$ is linearly normal and not contained in a quadric hypersurface. Then there exists a monad of type
\begin{equation}\tag{M}
\xymatrix@R=\linhaxylabels@C=\colunaxylabels{0\ar[r]&(L^\lor)^a\ar[r]^-{f}&\OO_{X}^{\,b}\ar[r]^-{g}&L^c\ar[r]&0}
\end{equation}
if and only if one of the following conditions holds:
\begin{enumerate}[(i)]
\item  $b \geq a + c$ and $b \geq 2c+n-1$, \label{genmonad1}
\item  $b \geq a+c+n$. \label{genmonad2}
\end{enumerate}
If so, there actually exists a monad with the map $f$ degenerating in expected codimension $b-a-c+1$. Furthermore, $g$ can be defined by a matrix whose entries are global sections of $L$ that span a subspace of $V$ whose dimension is ${\min\big(b-2c+2,\dim V\big)}$.
\end{theorem}
\begin{proof}
The proof of the existence of a monad of type \eqref{genmonad} follows again from Lemma \ref{existencelemma}. Let us check that at least one of conditions \eqref{monadFlo1} or \eqref{monadFlo2} is necessary. Let ${N=\dim V-1}$ and denote ${\PP^N=\PP(V)}$.

Suppose that there is a monad
\[
\xymatrix@R=\linhaxylabels@C=\colunaxylabels{0\ar[r]&(L^\lor)^a\ar[r]^-{f}&\OO_{X}^{\,b}\ar[r]^-{g}&L^c\ar[r]&0.}
\]
Let $A$ and $B$ be matrices defining $f$ and $g$, respectively. Since the entries of both matrices are elements of $H^0(L)$ and $X'$ is linearly normal, we can choose linear forms on $\PP^N$ to represent them,  so the entries in the product $BA$ can be regarded as elements of $H^0\big(\OO_{\PP^N}(2)\big)$. Since $X'$ is not cut out by any quadric, $BA$ is zero on $\PP^N$ yielding a complex
\[
\xymatrix@R=\linhaxylabels@C=\colunaxylabels{\OO_{\PP^N}(-1)^a\ar[r]^-{\tilde{f}}&\OO_{\PP^N}^{\,b}\ar[r]^-{\tilde{g}}&\OO_{\PP^N}(1)^c}.
\]
Furthermore, denoting by $Z_A$ and $Z_B$ the degeneracy loci in $\PP^N$ of $A$ and $B$, respectively, we know that ${\dim (Z_A\cap X')\le n-1}$ and $Z_B$ does not intersect $X'$. Therefore their dimensions satisfy ${\dim Z_A\le N-1}$ and ${\dim Z_B\le N-n-1}$. We can consider a general subspace $\PP^n$ that does not meet $Z_B$ and also satisfies ${\dim (Z_A\cap \PP^n)\le n-1}$. So if we consider the complex
\[
\xymatrix@R=\linhaxylabels@C=\colunaxylabels{\OO_{\PP^n}(-1)^a\ar[r]^-{\hat{f}}&\OO_{\PP^n}^{\,b}\ar[r]^-{\hat{g}}&\OO_{\PP^n}(1)^c},
\]
also defined by the matrices $A$ and $B$, we see that $\hat{f}$ is injective and $\hat{g}$ is surjective, so we have a monad on $\PP^n$ and by Theorem \ref{TeoFlo} at least one of conditions \eqref{monadFlo1} and \eqref{monadFlo2} is satisfied.
\end{proof}

\begin{Example}
In \cite{MMS14} the second and third authors presented a collection of examples of monads on Segre varieties. Using the same approach, we can think of similar examples of monads of some varieties that are cut out by quadrics, such as the Grassmannian. The simplest case that is not a hypersurface is ${\GG(2,5)}$, the Grassmannian that parametrises planes in the projective space $\PP^5$, which is embedded in $\PP^{19}$ with Pl\"{u}cker coordinates ${[X_{j_0j_1j_2}]_{0\le j_0<j_1<j_2\le5}}$ satisfying
\begin{equation}\label{Grassmannianequations}
\sum_{s=0}^3(-1)^{s}X_{j_0j_1l_s}X_{l_0\cdots \widehat{l_s}\cdots l_3}=0
\end{equation}
for ${0\le j_0<j_1\le 5}$ and ${0\le l_0<l_1<l_2<l_3\le 5}$, where ${X_{i_0i_1i_2}=(-1)^\sigma X_{i_{\sigma_0}i_{\sigma_1}i_{\sigma_2}}}$, for any permutation $\sigma$, and ${X_{i_0i_1i_2}=0}$ if there are any repeated indices. One of these quadrics is
\begin{multline*}
X_{012}X_{345} - X_{013}X_{245} + X_{014}X_{235} - X_{015}X_{234}=\\
  \tfrac{1}{4} \big( (X_{012}+X_{345})^2 - (X_{012}-X_{345})^2
    - (X_{013}+X_{245})^2 + (X_{013}-X_{245})^2\\
  + (X_{014}+X_{235})^2 - (X_{014}-X_{235})^2
    - (X_{015}+X_{234})^2 + (X_{015}-X_{234})^2\big),
\end{multline*}
obtained by using the sextuple ${(j_0,j_1,l_0,l_1,l_2,l_3)=(0,1,2,3,4,5)}$ in \eqref{Grassmannianequations}. Now, for any pair $(a,b)$, with ${1\le a<b\le5}$, consider the linear forms
\[{u_{ab}:=X_{0ab}+X_{i_1i_2i_3}}\text{ and }{v_{ab}:=X_{0ab}-X_{i_1i_2i_3}},\]
where $i_1$, $i_2$, and $i_3$, are the unique integers satisfying ${\{a,b,i_1,i_2,i_3\}=\{1,2,3,4,5\}}$ and ${i_1<i_2<i_3}$. Then the twenty forms in $\{u_{ab},v_{ab}\}_{1\le a<b\le5}$ form a new basis of the coordinate ring of $\PP^{19}$ and the above quadric can be rewritten as
\[
\tfrac{1}{4}\big(u_{12}^{\,2} - v_{12}^{\,2} - u_{13}^{\,2} + v_{13}^{\,2} + u_{14}^{\,2} - v_{14}^{\,2}
  - u_{15}^{\,2} + v_{15}^{\,2}\big).
\]
So if seven of the eight linear forms occurring in this quadric vanish at a point of ${\GG(2,5)}$ so does the eighth. Similarly, using ${(0,3,1,2,4,5)}$, ${(0,4,1,2,3,5)}$, and ${(0,5,1,2,3,4)}$ for ${(j_0,j_1,l_0,l_1,l_2,l_3)}$ in \eqref{Grassmannianequations}, we see that ${\GG(2,5)}$ is also cut out by
\begin{align*}
\tfrac{1}{4}\big(- u_{13}^{\,2} + v_{13}^{\,2} + u_{23}^{\,2} - v_{23}^{\,2} + u_{34}^{\,2} 
  - v_{34}^{\,2} - u_{35}^{\,2} + v_{35}^{\,2}\big),\\
\tfrac{1}{4}\big(- u_{14}^{\,2} + v_{14}^{\,2} + u_{24}^{\,2} - v_{24}^{\,2} - u_{34}^{\,2}
  + v_{34}^{\,2} - u_{45}^{\,2} + v_{45}^{\,2}\big),\\
\intertext{and}
\tfrac{1}{4}\big(- u_{15}^{\,2} + v_{15}^{\,2} + u_{25}^{\,2} - v_{25}^{\,2} - u_{35}^{\,2}
  + v_{35}^{\,2} + u_{45}^{\,2} - v_{45}^{\,2}\big).
\end{align*}
Therefore, the sixteen linear forms ${u_{23},\ldots,u_{45},v_{12},\ldots,v_{45}}$ cannot simultaneously vanish at a point of the Grassmannian, otherwise so would the remaining four $u_{12}$, $u_{13}$, $u_{14}$, and $u_{15}$. So let us write
\begin{align*}
w_1&=u_{23}, & w_2&=u_{24}, & w_3&=u_{25}, & w_4&=u_{34}, & w_5&=u_{35},
  & w_6&=u_{45}, & w_7&=v_{12}, & w_8&=v_{13},\\
w_9&=v_{14}, & w_{10}&=v_{15}, & w_{11}&=v_{23}, & w_{12}&=v_{24},
  & w_{13}&=v_{25}, & w_{14}&=v_{34}, & w_{15}&=v_{35}, & w_{16}&=v_{45}.
\end{align*}

Let ${k\ge1}$ and let ${A_1,A_2\in M_{(k+7)\times k}(S)}$ and ${B_1,B_2\in M_{k\times(k+7)}(S)}$ be the matrices with entries in ${S:=K[X_{012},\ldots,X_{345}]}$, given by
\begin{gather*}
A_1=
\begin{bmatrix}
w_8\\
\vdots&\ddots\\
w_1&&w_8\\
&\ddots&\vdots\\
&&w_1
\end{bmatrix},
\quad
A_2=
\begin{bmatrix}
w_{16}\\
\vdots&\ddots\\
w_9&&w_{16}\\
&\ddots&\vdots\\
&&w_9
\end{bmatrix},\\[1ex]
B_1=
\begin{bmatrix}
w_1&\cdots&w_8\\
&\ddots&&\ddots\\
&&w_1&\cdots&w_8
\end{bmatrix},
\quad\mbox{and}\quad
B_2=
\begin{bmatrix}
w_9&\cdots&w_{16}\\
&\ddots&&\ddots\\
&&w_9&\cdots&w_{16}
\end{bmatrix},
\end{gather*}
and note that ${B_1A_2=B_2A_1}$. Let $A$ and $B$ be the matrices
\begin{equation}\label{Segremonadmatrices}
A=\begin{bmatrix}-A_2\\A_1\end{bmatrix}\qquad\textrm{and}\qquad
 B=\begin{bmatrix}B_1&B_2\end{bmatrix},
\end{equation}
and let
\begin{equation*}
\xymatrix@C=\colunaxylabels{0\ar[r]&\OO_{\GG(2,5)}(-1)^k
 \ar[r]^-\alpha&\OO_{\GG(2,5)}^{\; 2k+14}
 \ar[r]^-\beta&\OO_{\GG(2,5)}(1)^k\ar[r]&0}
\end{equation*}
be the sequence with maps $\alpha$ and $\beta$ defined by matrices $A$ and $B$, respectively. Now $A$ and $B$ fail to have maximal rank $k$ if and only if ${w_1,\ldots,w_{16}}$ are all zero, which, as we have seen, cannot happen in the Grassmannian variety. In particular, $\alpha$ is injective and $\beta$ is surjective, and since ${BA=0}$, this sequence yields a monad. We can reduce the exponent of the middle term in the monad, by using the method we described in the proof of Theorem \ref{maintheorem}, combined with this construction. Let  ${\Lambda \subset \PP^{19}}$ be the projective subspace defined by the following ten linear forms
\begin{align*}
&X_{012} - X_{345}; \quad X_{013} - X_{245}; \quad X_{014} - X_{235};
  \quad X_{015} - X_{234};  \\
&X_{023} - X_{145}; \quad X_{024} - X_{135}; \quad X_{025} - X_{034};
  \quad X_{035} - X_{123};  \\
&X_{045} + X_{134} + X_{124}; \quad \text{and}
  \quad X_{045} + 2X_{134} - 3X_{124} - 5X_{125} + 7X_{012} + 11X_{013}.
\end{align*}
With the help of a computer algebra system such as Macaulay \cite{GS}, we can check that $\Lambda$ is disjoint from $\GG(2,5)$, so if ${w_1',\ldots,w_{10}'}$ are linear forms that complete a basis for the coordinate ring of $\PP^{19}$, we can use them to construct matrices analogous to $A$ and $B$ above and obtain a monad
\begin{equation*}
\xymatrix@C=\colunaxylabels{0\ar[r]&\OO_{\GG(2,5)}(-1)^k
 \ar[r]^-\alpha&\OO_{\GG(2,5)}^{\; 2k+8}
 \ar[r]^-\beta&\OO_{\GG(2,5)}(1)^k\ar[r]&0.}
\end{equation*}
\end{Example}

\section{Simplicity}
Recall that a vector bundle $E$ is said to be simple if its only endomorphisms
are the homotheties, i.e. $\Hom(E,E)=\CC$.
The cohomology of a monad on $\PP^N$ of type \eqref{monadFlo} is known to be simple when it has rank $N-1$ (see \cite{AO94}). Moreover, every instanton bundle on the hyperquadric $Q^{2l+1}\subset\PP^{2l+2}$ is simple (\cite{CMR09}).

We next address the problem of the simplicity of the cohomology of monads on projective varieties of the form \eqref{genmonad}.

\begin{proposition}\label{simpleprop}
Let $X$ be a variety of dimension $n$ and let $L$ be a line bundle on $X$. Suppose there is a linear system ${V\subseteq H^0(L)}$, with no base points, ${\dim V\ge3}$, defining a morphism ${X\to\PP(V)}$ such that the ideal sheaf of its image $X'$ satisfies
\[
h^1\big(\mathcal{I}_{X'}(-1)\big)=h^2\big(\mathcal{I}_{X'}(-1)\big)=
  h^2\big(\mathcal{I}_{X'}(-2)\big)=h^3\big(\mathcal{I}_{X'}(-2)\big)=0.
\]
Let $a$ and $b$ be integers such that
\[
{\max\{n+1,a+1\}\le b\le \dim V}.
\]
Then there exists a monad
\begin{equation}\label{Simplemonad}
\xymatrix@R=\linhaxylabels@C=\colunaxylabels{0\ar[r]&(L^\vee)^a\ar[r]^-{f}&\OO_{X}^{\,b}\ar[r]^-{g}&L\ar[r]&0.}
\end{equation}
whose cohomology sheaf is simple.

Moreover, when $b$ is minimal, that is $b=n+1$, then any monad of type \eqref{Simplemonad} has a simple cohomology sheaf.
\end{proposition}

\begin{proof}
Let ${N=\dim V-1}$ and write $\PP^N$ for $\PP(V)$. Since ${b\ge \max\{n+1,a+1\}}$, Lemma \ref{existencelemma} guarantees the existence of a monad of type \eqref{Simplemonad}. Moreover, since ${b\le \dim V}$, we can choose linearly independent linear forms for the matrix that represents $g$. Consider the display
\begin{equation*}
\xymatrix@R=\linhaxylabels@C=\colunaxylabels{
&&0\ar[d]&0\ar[d]\\
0\ar[r]&(L^\vee)^a\ar[r]\ar@{=}[d]&K\ar[r]\ar[d]&E\ar[r]\ar[d]&0\\
0\ar[r]&(L^\vee)^a\ar[r]^-{f}&\OO_{X}^{\,b}\ar[r]\ar[d]^-{g}&Q\ar[r]\ar[d]&0\\
&&L\ar[d]\ar@{=}[r]&L\ar[d]\\
&&0&0
}
\end{equation*}
Dualising the first column and tensoring with ${K}$ we get
\begin{equation}\label{SimplelowrankmonaddisplaytensorK}
\xymatrix@R=\linhaxy@C=\colunaxy{0\ar[r]&K\otimes L^\vee \ar[r]&{K}^{b}\ar[r]&K\otimes K^\vee\ar[r]&0.}
\end{equation}

We claim that $K$ is simple, i.e.\ ${h^0(K\otimes K^\vee)=1}$. To see this, we first observe that, by construction, $L\cong \morphismtopn^*\OO_{\PP^N}(1)$, where ${\morphismtopn:X\to\PP^N}$ is the morphism given by $L$. So, considering $\OO_{X'}$ as a sheaf over $\PP^N$, we have $\morphismtopn_*L\cong \OO_{X'}(1)$, and therefore $\morphismtopn_*L^\vee\cong \OO_{X'}(-1)$ and $\morphismtopn_*(L^\vee\otimes L^\vee)\cong \OO_{X'}(-2)$. Consider the exact sequence on $\PP^N$
\[
\xymatrix@R=\linhaxy@C=\colunaxy{0\ar[r]&\mathcal{I}_{X'}(-1)\ar[r]&{\OO_{\PP^N}}(-1)\ar[r]&\morphismtopn_*L^\vee\ar[r]&0.}
\]
Taking cohomology, we get that ${h^0(L^\vee)=h^1(L^\vee)=0}$ from the vanishing of the groups ${H^1\big(\mathcal{I}_{X'}(-1)\big)}$, ${H^2\big(\mathcal{I}_{X'}(-1)\big)}$, and ${H^i\big(\OO_{\PP^N}(-1)\big)}$. Now if we tensor the first column of the display by ${L^\vee}$ and take cohomology, we get $h^1(K\otimes L^\vee)=h^0(\OO_X)=1$. Note also that ${H^0(g):H^0(\OO_{X}^{\,b})\to H^0(L)}$ is injective, since the linear forms we chose to construct the matrix for $g$ are linearly independent, hence ${h^0(K)=0}$. Therefore we get an injective morphism
\begin{equation*}
\xymatrix@R=\linhaxy@C=\colunaxy{0\ar[r]&H^0(K\otimes K^\vee)\ar[r]&H^1(K\otimes L^\vee)}
\end{equation*}
induced by the exact sequence in \eqref{SimplelowrankmonaddisplaytensorK}, and we get ${h^0(K\otimes K^\vee)=1}$, as we wished.

We now consider the exact sequence
\[
\xymatrix@R=\linhaxy@C=\colunaxy{0\ar[r]&\mathcal{I}_{X'}(-2)\ar[r]&{\OO_{\PP^N}}(-2)\ar[r]&\morphismtopn_*(L^\vee\otimes L^\vee)\ar[r]&0}
\]
and take cohomology to get ${h^1(L^\vee\otimes L^\vee)=h^2(L^\vee\otimes L^\vee)=0}$, from the vanishing of the groups ${H^2\big(\mathcal{I}_{X'}(-2)\big)}$, ${H^3\big(\mathcal{I}_{X'}(-2)\big)}$, and ${H^i\big(\OO_{\PP^N}(-2)\big)}$, since ${N\ge2}$. We dualise the first row in the display and tensor by $E$ to obtain
\begin{equation}\label{SegrePlPmmonaddisplaytensorE}
\xymatrix@R=\linhaxy@C=\colunaxy{0\ar[r]&E\otimes E^\vee \ar[r]&E\otimes K^\vee\ar[r]&E\otimes L^a\ar[r]&0,}
\end{equation}
which induces an injective morphism
\begin{equation*}
\xymatrix@R=\linhaxy@C=\colunaxy{0\ar[r]&H^0(E\otimes E^\vee) \ar[r]&H^0(E\otimes K^\vee).}
\end{equation*}
By dualising the first column and tensoring by ${L^\vee}$, we may take cohomology and see that\linebreak ${h^0(K^\vee\otimes L^\vee)=h^1(K^\vee\otimes L^\vee)=0}$, for ${h^1(L^\vee\otimes L^\vee)=h^2(L^\vee\otimes L^\vee)=0}$ and ${h^0(L^\vee)=h^1(L^\vee)=0}$, as we saw above. Now tensoring the first row of the display by $K^\vee$ and taking cohomology we get ${h^0(K^\vee\otimes E)=h^0(K^\vee\otimes K)=1}$. Therefore ${h^0(E^\vee\otimes E)=1}$, i.e.\ $E$ is a simple sheaf.

Finally, given any monad of type \eqref{Simplemonad} with $b=n+1=\dim X+1$, the entries of the matrix defining $g$ must be linearly independent, otherwise it would not have maximal rank and $g$ would not be a surjective morphism. Since the linear independence of these linear forms is a key step in the beginning of the proof, we see that in this case, any monad of type \eqref{Simplemonad} has a simple cohomology sheaf.
\end{proof}

The next example shows that the statement in Proposition \ref{simpleprop} is accurate, that is  there are monads of type \eqref{Simplemonad} whose cohomology is not simple.

\begin{Example}
Consider the monad over the quadric $X\subset \PP^3$ embedded in $\PP^9$ by $L=\OO_X(2)$,
\[
\xymatrix@R=\linhaxylabels@C=\colunaxylabels{
0 \ar[r] & \OO_{X}(-2) \ar[r]^-{f}_-{A} & \OO_{X}^{\,5} \ar[r]^-{g}_-{B} & \OO_{X}(2) \ar[r] & 0,
}
\]
where
\[
B=
\begin{bmatrix}
x_0^{\,2} & x_1^{\,2} & x_2^{\,2} & x_3^{\,2} & x_3^{\,2}
\end{bmatrix}, 
\quad 
A=
\begin{bmatrix}
-x_3^{\,2} & -x_2^{\,2} & x_1^{\,2} & x_0^{\,2} & 0 
\end{bmatrix}^T,
\]
and $x_i$ are the coordinates in $\PP^3$ such that $X$ is defined by the form $x_0^{\,2}+x_1^{\,2}+x_2^{\,2}+x_3^{\,2}$. Then ${\max\{3,2\}\le b\le h^0\big(\OO_X(2)\big)=10=N+1}$, however $E$ is not simple.
In fact, first note that $K=\ker g$ is not simple since it admits the endomorphism
\[
\varphi: (f_1,f_2,f_3,f_4,f_5) \mapsto (f_1,f_2,f_3,f_5,f_4),
\]
clearly not a homothety of $K$: if ${f_4\neq f_5}$ then $\varphi(f_1,f_2,f_3,f_4,f_5)$ is not a multiple of $(f_1,f_2,f_3,f_4,f_5)$. Therefore the endomorphism induced on $E\cong K/ \im f$ by $\varphi$ is not a homothety of $E$ (the class of a $5$-uple of the same form is not mapped into a multiple of itself).
\end{Example}

\section{Vector bundles of low rank}
In this section we characterise monads whose cohomology is a vector bundle of rank lower than the dimension of $X$ and, in particular, we restrict to the case when $X$ is non-singular. Moreover, we will deal with the problem of simplicity and stability of this particular case.

Generalising Floystad's result, we start by proving the following theorem.
\begin{theorem}\label{lowerrankth}
Let $X$ be a non\h{singular}, \mbox{$n$-dimensional}, projective variety, embedded in $\PP^N$ by a very ample line bundle $L$. Let $M$ be a monad as in \eqref{genmonad} and $E$ its cohomology. If $E$ is a vector bundle of rank lower than $n$, then ${n=2k+1}$, ${\rk E = 2k}$ and the monad is of type
\begin{equation}\label{bundlemonad2}
\xymatrix@R=\linhaxylabels@C=\colunaxylabels{0\ar[r]&(L^\lor)^c\ar[r]^-{f}&\OO_{X}^{\,2k+2c}\ar[r]^-{g}&L^c\ar[r]&0.}
\end{equation}
Conversely, for each odd dimensional variety $X$ with an associated ACM embedding given by a line bundle $L$ and for each $c \geq 1$ there exists a vector bundle which is cohomology of a monad of type \eqref{bundlemonad2}.
\end{theorem}
\begin{proof}
Suppose we have a vector bundle $E$ of rank lower than $\dim X = n$ which is the cohomology of a monad $M$ of type \eqref{genmonad}. Then its dual $E^\lor$ is the cohomology of the dual monad $M^\lor$. Since both $E$ and $E^\lor$ are vector bundles which do not satisfy condition \eqref{genmonad2} of Theorem \ref{maintheorem}, we must have
\[
b \geq 2c + n - 1 \:\:\:\mbox{and}\:\:\: b \geq 2a +n -1.
\]
On the other hand, the hypothesis $\rk E < n$ implies that
\[
b \leq a+c+n-1.
\]
Combining the three inequalities we get that
\[
a = c \:\:\:\mbox{and}\:\:\: b = 2c + n -1.
\]
Then the monad $M$ is of type
\[
\xymatrix@R=\linhaxylabels@C=\colunaxylabels{0\ar[r]&(L^\lor)^c\ar[r]^-{f}&\OO_{X}^{\,2c + n - 1}\ar[r]^-{g}&L^c\ar[r]&0,}
\]
therefore $\rk E = n-1$ which implies that $c_n(E)=0$.

Hence, since the Chern polynomial of $\OO_X$ is $c_t(\OO_X)=1$ (for $X$ is non\h{singular}), we have
\[
c_t(E)=\frac{1}{(1-lt)^c  (1+lt)^c}=(1+l^2t^2+l^4t^4+\cdots)^c,
\]
where $l$ denotes $c_1(L)$.  If $n=2k$, for some $k\in\ZZ$, then $c_{2k}(E)=\alpha_{2k}l^{2k}$, where $\alpha_{2k}>0$ is the binomial coefficient of the expansion of the series of $c_t(E)$. Observe that $l^{2k}=c_{2k}(L^{2k})$ and, by the projection formula (see \cite{Ful98}, Theorem 3.2 (c)), this Chern class cannot be zero, contradicting the assertion above. So we conclude that $n$ is odd and that the monad is of type \eqref{bundlemonad2}.

Conversely, for any $c\geq 1$ and $(2k+1)$-dimensional variety $X$, there exists a monad of type \eqref{bundlemonad2} whose cohomology is a vector bundle $E$ of rank $2k$ constructed using the technique described in the proof of Lemma \ref{existencelemma}.
\end{proof}

\subsection*{Minimal rank bundles defined using ``many'' global sections}
In Theorem \ref{maintheorem} we showed that the morphism $g$ in \eqref{genmonad} can be defined by a matrix $B$ whose entries are global sections of $L$ that span a subspace of $H^0(L)$ of dimension ${\min\big(b-2c+2,h^0(L)\big)}$. This was done by giving an example of such a matrix, but surely there are others. Moreover, the dimension of the subspace spanned by the entries of these matrices can be bigger as we shall see in the following examples.

Take the quadric hypersurface ${Q_3 \subset \PP^4}$, defined by the equation ${x_0^{\,2} + x_1^{\,2}+ x_2^{\,2} +x_3^{\,2} + x_4^{\,2}=0}$,  and $L = \OO_{Q_3}(2)$. Following the techniques used in Section \ref{SectionExistence} to construct a monad, we are able to obtain
\begin{equation}\label{ex-monquad}
\xymatrix@R=\linhaxylabels@C=\colunaxylabels{\OO_{Q_3}(-2)^2\ar[r]^-{A} & \OO_{Q_3}^{\,6}\ar[r]^-{B} & \OO_{Q_3}(2)^2,}
\end{equation}
where
\begin{equation}
A =
\begin{bmatrix}
-x_2^{\,2} & -x_3^{\,2} \\
0 & -x_2^{\,2} \\
 -x_3^{\,2} & 0\\
x_0^{\,2} & x_1^{\,2} \\
 0 & x_0^{\,2} \\
  x_1^{\,2} & 0
\end{bmatrix}
\quad\text{and}\quad
B = 
\begin{bmatrix}
x_0^{\,2} & x_1^{\,2} & 0 & x_2^{\,2} & x_3^{\,2} & 0\\
0 & x_0^{\,2} & x_1^{\,2} & 0 & x_2^{\,2} & x_3^{\,2}
\end{bmatrix}.
\end{equation}
We have $BA = 0$, and $A$ and $B$ have maximal rank when evaluated at every point of $Q_3$. Indeed, the rank of both $A$ and $B$ is not maximal only when evaluated at the point $(0:0:0:0:1) \in \PP^4$, that does not belong to the quadric.

In order to use more global sections in the matrices defining the monad, we could simply ``add another diagonal" whose entries involve an additional global section.
Unfortunately, this method will increase the rank of the sheaf. For example, take the monad
\begin{equation}
\xymatrix@R=\linhaxylabels@C=\colunaxylabels{\OO_{Q_3}(-2)^2\ar[r]^-{A'}&\OO_{Q_3}^{\,7}\ar[r]^-{B'}&\OO_{Q_3}(2)^2}
\end{equation}
given by the matrices
\begin{equation}
A' = 
\begin{bmatrix}
-x_2^{\,2} -x_3^{\,2} & -x_4^{\,2} \\
 -x_2^{\,2} &  -x_3^{\,2} \\
 0 & -x_2^{\,2} -x_4^{\,2}\\
x_0^{\,2} + x_1^{\,2} & 0 \\
 x_0^{\,2} & x_1^{\,2} \\
0 &  x_0^{\,2}\\
0 &  x_1^{\,2}
\end{bmatrix}
\quad\text{and}\quad
B' = 
\begin{bmatrix}
x_0^{\,2} & x_1^{\,2} & 0 & x_2^{\,2} & x_3^{\,2} & x_4^{\,2}& 0\\
0 & x_0^{\,2} & x_1^{\,2} & 0 & x_2^{\,2} & x_3^{\,2} & x_4^{\,2}
\end{bmatrix}
\end{equation}
with maximal rank evaluated at each point of $Q_3$. The cohomology of this monad is a rank $3$ vector bundle on the quadric.

Therefore, our goal is to construct examples of minimal rank vector bundles whose monads are defined by matrices using a number of independent global sections of $L$ strictly bigger than $\dim X +1$. Indeed, the monads obtained this way cannot be the pullback of some monad over a projective space via a finite morphism (as described in Remark \ref{rmk-pullback}).

In the following two examples we will achieve such a goal in the particular case of the quadric considered above. However, the technique is easily reproducible for other varieties. The key point is to consider two matrices such that the union of their respective standard determinantal varieties does not intersect the base variety.

We get such examples by slightly modifying the matrices $A$ and $B$. Consider a monad of type \eqref{ex-monquad} but defined by the matrices
\begin{equation}
A_1 = 
\begin{bmatrix}
-x_2^{\,2} & -x_3^{\,2} \\
0 & -x_2^{\,2} \\
 -x_3 x_4 & 0\\
x_0^{\,2} & x_1^{\,2} \\
 0 & x_0^{\,2} \\
  x_1^{\,2} & 0
\end{bmatrix}
\quad\text{and}\quad
B_1 = 
\begin{bmatrix}
x_0^{\,2} & x_1^{\,2} & 0 & x_2^{\,2} & x_3^{\,2} & 0\\
0 & x_0^{\,2} & x_1^{\,2} & 0 & x_2^{\,2} & x_3 x_4
\end{bmatrix}.
\end{equation}
Then, $B_1A_1 = 0$, and both $A_1$ and $B_1$ have maximal rank at every point of $\PP^4$ except at points $(0:0:0:1:0)$ and $(0:0:0:0:1)$, neither belonging to the quadric.

It is possible to insert an additional global section in the previous matrices, by considering, for example, the monad defined by the matrices
\begin{equation}
A_2 = 
\begin{bmatrix}
-x_2^{\,2} & -x_3^{\,2} \\
0 & -x_2^{\,2} \\
 -x_3 x_4 & 0\\
x_0^{\,2} & x_1^{\,2} \\
 0 & x_0^{\,2} \\
  x_1^{\,2} +x_1 x_4& 0
\end{bmatrix}
\quad\text{and}\quad
B_2 = 
\begin{bmatrix}
x_0^{\,2} & x_1^{\,2} & 0 & x_2^{\,2} & x_3^{\,2} & 0\\
0 & x_0^{\,2} & x_1^{\,2} + x_1 x_4 & 0 & x_2^{\,2} & x_3 x_4
\end{bmatrix}.
\end{equation}
Again, $B_2A_2 = 0$, and $A_2$ and $B_2$ have maximal rank when evaluated at all points of the projective space except at $(0:0:0:1:0)$, $(0:0:0:0:1)$ and $(0:1:0:0:-1)$, that do not belong to the quadric.

As we wanted, in both examples we used a number of global sections strictly bigger than ${\dim Q_3 +1}$; it would be interesting to determine all the possible matrices obtained with this technique, once fixed the base variety and the monad.

\subsection*{Simplicity and stability}
We note that it is straightforward to construct examples of vector bundles on $X$, with $\Pic X = \ZZ$, satisfying properties of simplicity and stability. In fact, it is enough to consider, as observed in Remark \ref{rmk-pullback}, $\dim X +1$ generic sections of $L$ in order to get a finite morphism $\varphi: X \rightarrow \PP^{2n+1}$. Using the \emph{flatness miracle} and the projection formula, it is possible to prove that $\varphi_* \OO_X$ is locally free and, moreover, $\varphi_* \OO_X = \bigoplus_{i=0}^p \OO_{\PP^{2n+1}}(-a_i)$, for some positive $p$ and non negative $a_i$'s (see \cite{BPV84}, Lemma I.17.2).
Finally, using once again the projection formula (to the cohomology bundle) as well as commutativity of the tensor product with the pullback, we can conclude that the pullback of a simple (respectively stable) bundle $E$ on $\PP^{2n+1}$ is a simple (respectively stable) bundle on the projective variety $X$.

Nevertheless, we always have the following property.

\begin{theorem}
Let $X$ be a variety of dimension $n$ and let $L$ be a line bundle on $X$. Suppose there is a linear system ${V\subseteq H^0(L)}$, with no base points, defining a morphism ${X\to\PP(V)}$ whose image ${X'\subset\PP(V)}$ satisfies ${h^2\big(\mathcal{I}_{X'}(-1)\big)=0}$ and at least one of the following conditions:
\begin{enumerate}
\item $X'$ is a projective ACM variety;
\item $X'$ is linearly normal and is not contained in a quadric hypersurface.
\end{enumerate}
Suppose in addition that there is a monad of type \eqref{bundlemonad2} over $X$ whose cohomology sheaf $E$ is locally free. Then $H^0(E)=0$.
\end{theorem}

\begin{proof}
From the hypotheses, we see that $X$ satisfies the conditions in Theorem \ref{maintheorem} or Theorem \ref{notcutbyquadricstheorem}. A monad of type \eqref{bundlemonad2} over $X$ admits the following display:
\begin{equation*}
\xymatrix@R=\linhaxylabels@C=\colunaxylabels{
&&0\ar[d]&0\ar[d]\\
0\ar[r]&(L^\vee)^c\ar[r]\ar@{=}[d]&K\ar[r]\ar[d]&E\ar[r]\ar[d]&0\\
0\ar[r]&(L^\vee)^c\ar[r]^-{f}&\OO_{X}^{\, 2k+2c}\ar[r]\ar[d]^-{g}&Q\ar[r]\ar[d]&0\\
&&L^c\ar[d]\ar@{=}[r]&L^c\ar[d]\\
&&0&0
}
\end{equation*}
Taking cohomology on the exact sequence 
\[
\xymatrix@R=\linhaxy@C=\colunaxy{0\ar[r]&\mathcal{I}_{X'}(-1)\ar[r]&{\OO_{\PP^N}}(-1)\ar[r]&\morphismtopn_*L^\vee\ar[r]&0,}
\]
we get that ${h^0(L^\vee)=h^1(L^\vee)=0}$, since ${h^1\big(\mathcal{I}_{X'}(-1)\big)=h^2\big(\mathcal{I}_{X'}(-1)\big)=0}$. Therefore, taking cohomology on the first row of this display we have that $H^0(E) = H^0(K)$. Let us suppose that ${H^0(K) \ne 0}$, and let ${\delta=h^0(K)}$. Applying Lemma 1.6 in \cite{AMS16}, we see that $K \simeq K' \oplus \OO_X^{\delta}$, since $K^\vee$ is an \mbox{$(L^\vee,\OO_X)$-Steiner} bundle (see Definition 1.3 in \cite{AMS16}). Therefore, the matrix defining $g$, with a suitable change of variables, may be assumed to have $\delta$ zero columns. So, again by Lemma 1.6,  $(K')^\vee$ is itself an \mbox{$(L^\vee,\OO_X)$-Steiner} bundle, sitting on a short exact sequence
\[
\xymatrix@R=\linhaxy@C=\colunaxy{0\ar[r]&(L^\vee)^c\ar[r]&\OO_{X}^{\,2k+2c-\delta}\ar[r]&(K')^\vee\ar[r]&0.}
\]
Dualising this, we get
\[
\xymatrix@R=\linhaxy@C=\colunaxy{0\ar[r]&K'\ar[r]&\OO_{X}^{\,2k+2c-\delta}\ar[r]&L^c\ar[r]&0}
\]
with $H^0(K')=0$. Therefore, we would get a new monad, whose cohomology might be a sheaf, defined as
\[
\xymatrix@R=\linhaxy@C=\colunaxy{0\ar[r]&(L^\lor)^c\ar[r]&\OO_{X}^{\,2k+2c-\delta}\ar[r]&L^c\ar[r]&0.}
\]
But this contradicts the conditions of existence of Theorem \ref{maintheorem} and Theorem \ref{notcutbyquadricstheorem}, thus proving the statement.
\end{proof}
\begin{corollary}
Every rank 2 vector bundle $E$ on a three dimensional ACM smooth projective variety $X$ with $\Pic(X)=\ZZ$, defined by a monad of type \eqref{bundlemonad2}, is stable.
\end{corollary}

\begin{proof}
The result follows directly from the previous theorem and the Hoppe's criterion for stability (see \cite{Hop84}, Theorem 12).
\end{proof}

\section{The set of monads and the moduli problem}

The existence part in Theorem \ref{maintheorem} is proved by explicitly constructing a monad on a given projective variety $X$. The construction therein does not, however, give an answer to the question of ``how many'' monads of type \eqref{genmonad} exist.
We would like to know more about the algebraic structure of the set of pairs of morphisms which define a monad over a projective variety. In the case of the projective space we prove the following.

\begin{theorem}\label{theorem2}
Let $a,b,c$ satisfy the conditions of Theorem \ref{TeoFlo}, and suppose that ${1\le c\le 2}$. Then for any surjective morphism ${g\in\Hom\big(\OO_\pn^{\,b},\opn(1)^c\big)}$ there is a morphism ${f\in\Hom\big(\opn(-1)^a,\OO_\pn^{\,b}\big)}$ yielding a monad of type \eqref{monadFlo}.

Furthermore, the set of pairs ${(f,g)\in\Hom\big(\opn(-1)^a,\OO_\pn^{\,b}\big)\times \Hom\big(\OO_\pn^{\,b},\opn(1)^c\big)}$ yielding such a monad is an irreducible algebraic variety.
\end{theorem}

\begin{proof}
Let ${g\in\Hom\big(\OO_\pn^{\,b},\opn(1)^c\big)}$ be a surjective morphism and let ${K_g:=\ker g}$. Then for any injective morphism ${f\in H^0\big(K_g(1)\big)^a}$, the pair $(f,g)$ yields a monad of type \eqref{monadFlo}. If we consider the exact sequence
\[
\xymatrix@R=\linhaxylabels@C=\colunaxylabels{0\ar[r]&K_g\ar[r]&\OO_\pn^{\,b}\ar[r]^-{g}&\opn(1)^c\ar[r]&0,}
\]
tensor by ${\opn(1)}$ and take cohomology we can see that ${h^0\big(K_g(1)\big)=b(n+1)-c\tbinom{n+2}{2}+h^1\big(K_g(1)\big)}$. Now following the arguments in the proof of Theorem 3.2 in \cite{CMR07} we see that $K_g$ is \mbox{$m$-regular} for any ${m\ge c}$. Therefore, since ${c\le2}$, $K_g$ is \mbox{$2$-regular}, i.e.\ ${h^1\big(K_g(1)\big)=0}$. Since an injective morphism ${f\in H^0\big(K_g(1)\big)^a}$ comes from a choice of $a$ independent elements in $H^0\big(K_g(1)\big)$, we wish to show that ${h^0\big(K_g(1)\big)\ge a}$, i.e. ${b(n+1)-c\tbinom{n+2}{2}-a\ge0}$. We can check that the conditions in Therorem \ref{TeoFlo} imply this inequality.

The irreducibility of the set of pairs ${(f,g)}$ that yield a monad of type \eqref{monadFlo} comes from the fact that the subset of surjective morphisms ${g\in\Hom\big(\OO_\pn^{\,b},\opn(1)^c\big)}$ is irreducible and when we consider the projection
\[
\xymatrix@R=\linhaxy@C=\colunaxy{\Hom\big(\opn(-1)^a,\OO_\pn^{\,b}\big)\times \Hom\big(\OO_\pn^{\,b},\opn(1)^c\big)
\ar[r]&\Hom\big(\OO_\pn^{\,b},\opn(1)^c\big),}
\]
its fibre at a point corresponding to the surjective morphism $g$ is the irreducible set of injective morphisms in ${H^0\big(K_g(1)\big)^a}$, which has fixed dimension ${a\left(b(n+1)-c\tbinom{n+2}{2}\right)}$.
\end{proof}

Before discussing the more general setting of monads on ACM smooth projective varieties we give an example of reducibility with $c=5$ on the projective space.

\begin{Example}
Consider the set of instanton bundles defined by a monad of the form
\begin{equation*}
\xymatrix@R=\linhaxylabels@C=\colunaxylabels{0\ar[r]&\OO_{\PP^{3}}(-1)^5\ar[r]^-{f}&\OO_{\PP^{3}}^{\,12}\ar[r]^-{g}&\OO_{\PP^{3}}(1)^5\ar[r]&0.}
\end{equation*}
It was proved in \cite{JMT15} that the moduli space of instanton sheaves of rank $2$ and charge $5$ is reducible. Furthermore, the set of pairs ${(f,g)\in\Hom\big(\OO_{\PP^3}(-1)^5,\OO_{\PP^3}^{\,12}\big)\times \Hom\big(\OO_{\PP^3}^{\,12},\OO_{\PP^3}(1)^5\big)}$ yielding such a monad is a reducible algebraic variety. \end{Example}

\subsection*{The general setting.}

When $X$ is a projective variety the general setting is the following. Let $X$ be a projective variety embedded on $\PP^N$ by a very ample line bundle $L$. Consider the set of all morphisms $g : \OO_X^{\,b} \rightarrow L^c$, described by the vector space $B^* \otimes C \otimes H^0(L)$, where $B$ and $C$ are, respectively, $k$-vector spaces of dimensions $b$ and $c$.

Denote ${\PP\big(B^* \otimes C \otimes H^0(L)\big)}$ by $\PP$ and consider the map
\begin{equation}\label{tautMor}
\xymatrix@R=\linhaxy@C=\colunaxy{\OO_\PP(-1) \ar[r] & 
  B^* \otimes C \otimes H^0(L) \otimes \OO_\PP}
\end{equation}
of sheaves over $\PP$, whose fiber at a point in $\PP\big(B^* \otimes C \otimes H^0(L)\big)$ corresponds to the natural inclusion. So, from \eqref{tautMor} we get a map
\[
\xymatrix@R=\linhaxy@C=\colunaxy{B \otimes H^0(L) \otimes \OO_\PP(-1) \ar[r] & 
  C \otimes H^0(L) \otimes H^0(L) \otimes \OO_\PP,}
\]
and hence also a map
\[
\xymatrix@R=\linhaxylabels@C=\colunaxylabels{
  B \otimes H^0(L) \otimes \OO_\PP(-1) \ar[r]^-{\varphi} & 
  C \otimes H^0(L \otimes L) \otimes \OO_\PP,}
\]
induced by the natural morphism ${H^0(L) \otimes H^0(L) \to H^0(L\otimes L)}$.

Now recall that $h^0(L) = N+1$ and suppose that $a,b,c$ are positive integers that satisfy the conditions of Theorem \ref{maintheorem}. Then the degeneracy locus
\[
Z = \left\{ g \in \PP\big(B^* \otimes C \otimes H^0(L)\big) \:|\: \rk_g (\varphi) \leq b(N+1)-a \right\}
\]
describes the set of morphisms $g$ in a short exact sequence
\[
\xymatrix@R=\linhaxylabels@C=\colunaxylabels{
  0 \ar[r] & K_g \ar[r] & \OO_{X}^{\,b}  \ar[r]^-{g} &  L^c \ar[r] & 0}
\]
such that $h^0(K_g \otimes L) \geq a$ and for which it is thus possible to construct a monad of type \eqref{genmonad}.
Note furthermore that
\[
\codim Z \leq a\big(c \tbinom{N+2}{2} - b(N+1) +a \big).
\]
Hence, whenever $Z$ is irreducible (for example when $\codim Z <0$) and $h^0(K_g)$ is constant for every morphism $g$ we see that the set of the pairs $(f,g)$ yielding a monad \eqref{genmonad} on $X$ is an irreducible algebraic variety. In this case, Theorem \ref{theorem2} can be extended to ACM varieties.

\subsection*{The moduli space of vector bundles of low rank.}
Let $X$ be an ACM smooth projective variety of odd dimension $2k+1$, for some $k\in\mathbb{N}$, with an embedding in $\PP^N$ by a very ample line bundle $L$ on $X$, where ${h^0(L)=N+1}$.

Consider the set $\mathcal{V}_{2k,c}$ of rank $2k$ vector bundles which are the cohomology of a monad of type
\begin{equation}\label{bundlemonad1}
\xymatrix@R=\linhaxylabels@C=\colunaxylabels{0\ar[r]&(L^\lor)^c\ar[r]^-{f}&\OO_{X}^{\,2k+2c}\ar[r]^-{g}&L^c\ar[r]&0,}
\end{equation}
with $1\leq c\leq 2$.

\begin{Remark}\label{mon-iso-vb}
Observe that the hypotheses in Corollary 1, \S 4 Chapter 2, in \cite{OSS80}, hold for monads defined by \eqref{bundlemonad1}. Hence the isomorphisms of monads of this type correspond bijectively to the isomorphisms of the corresponding cohomology bundles. In particular, the two categories are equivalent and we will not distinguish between their corresponding objects.
\end{Remark}

We want to construct a moduli space $M(\mathcal{V}_{2k,c})$ of vector bundles in $\mathcal{V}_{2k,c}$.
In order to do this we will use King's framework of moduli spaces of representations of finite dimensional algebras in \cite{Kin94}.

We first note that according to \cite{JP15}, Theorem 1.3, the category $\mathcal{M}_{k,c}$ of monads of type \eqref{bundlemonad1} is equivalent to the full subcategory $\mathcal{G}^{gis}_{k,c}$ of the category $\mathcal{R}(Q_{k,c})$ of representations \linebreak ${R=\big(\{\CC^c,\CC^{2k+2c},\CC^c\},\{A_i\}_{i=1}^{N+1},\{B_j\}_{i=1}^{N+1}\big)}$ of the quiver $Q_{k,c}^{\ }$ of the form
\begin{equation*}
\xymatrix@C=3em@R=0.2em{
\\
\bullet \ar@<1.5ex>[r] \ar@{}[r]|{\ldots} \ar@<-1.5ex>[r]_{(N+1)}&\bullet \ar@<1.5ex>[r] \ar@{}[r]|{\ldots} \ar@<-1.5ex>[r]_{(N+1)} &\bullet }
\end{equation*}
which are $(\sigma,\gamma)$-globally injective and surjective and satisfy
\begin{equation}\label{fg=0}
\sum(B_iA_j+B_jA_i)\otimes(\sigma_i\gamma_j)=0.
\end{equation}
Let us briefly recall here the definitions of $(\sigma,\gamma)$-globally injective and surjective (see \cite{JP15} for more details).
Given a monad as in \eqref{bundlemonad1}, choose bases $\gamma=(\gamma_1,\ldots,\gamma_{N+1})$ of $\Hom(L^\vee,\mathcal{O}_X)$ and $\sigma=(\sigma_1,\ldots,\sigma_{N+1})$ of $\Hom(\mathcal{O}_X,L)$. Set
\[
\alpha=\sum_{i=1}^{N+1}A_i\otimes\gamma_i\quad \text{and} \quad
  \beta=\sum_{j=1}^{N+1}B_j\otimes\sigma_j.
\]
The monad conditions of injectivity of $f$ and surjectivity of $g$ are reinterpreted in the language of the associated representation ${R=\big(\{\CC^c,\CC^{2k+2c},\CC^c\},\{A_i\}_{i=1}^{N+1},\{B_j\}_{i=1}^{N+1}\big)}$ in $\mathcal{G}^{gis}_{k,c}$ as, respectively, ${\alpha(P)=\sum_{i=1}^{N+1}A_i\otimes\gamma_i(P)}$ is injective and ${\beta(P)=\sum_{j=1}^{N+1}B_j\otimes\sigma_j(P)}$ is surjective, for all ${P\in X}$. In this case, we say that $R$ is $(\sigma,\gamma)$-globally injective and surjective. The monad condition ${g\circ f=0}$ is rewritten as in \eqref{fg=0}.

For the sake of simplicity, we will write $R=(c,2k+2c,c)$ when we refer to the representation ${R=(\{\CC^c,\CC^{2k+2c},\CC^c\},\{A_i\},\{B_j\})}$.
The notion of semistability for representations in $\mathcal{G}^{gis}_{k,c}$, as defined by King, is the following: a representation ${R=(c,2k+2c,c)}$ is \emph{$\lambda$-semistable} is there is a triple ${\lambda=(\lambda_1,\lambda_2,\lambda_3)\in\ZZ^3}$ such that
\begin{eqnarray*}
\langle(\lambda_1,\lambda_2,\lambda_3),(c,2k+2c,c)\rangle= 0, \\
\langle(\lambda_1,\lambda_2,\lambda_3),(a',b',c')\rangle\geq 0,
\end{eqnarray*}
for all subrepresentations ${R'=(a',b',c')}$ of the representation $R$ ($\langle\cdot,\cdot\rangle$ denotes the usual dot product). The representation is \emph{$\lambda$-stable} if the only subrepresentations $R'$ with ${\langle(\lambda_1,\lambda_2,\lambda_3),(a',b',c')\rangle= 0}$ are $R$ and $0$.

Moreover, by King's central result (\cite{Kin94}, Teorem 4.1), the existence of such a $\lambda$ guarantees the existence of a coarse moduli space for families of $\lambda$-semistable representations up to $S$-equivalence (two $\lambda$-semistable representations are $S$-equivalent if they have the same composition factors in the full abelian subcategory of $\lambda$-semistable representations).

Given the equivalences of the categories $\mathcal{M}_{k,c}$ and $\mathcal{G}^{gis}_{k,c}$, and after Remark \ref{mon-iso-vb}, we see that we can define a moduli space $\mathcal{M}(\mathcal{V}_{2k,c})$ whenever we can construct a moduli space of the abelian category $\mathcal{G}^{gis}_{k,c}$.

When $c=1$ we prove:

\begin{theorem}\label{modulispace}
There is a coarse moduli space $\mathcal{M}(\mathcal{V}_{2k,1})$ of $\lambda$-semistable vector bundles in $\mathcal{V}_{2k,1}$.
\end{theorem}
\begin{proof}
Let $R=(1,2k+2,1)$ be a representation in $\mathcal{G}^{gis}_{k,1}$, let $R'=(a',b',c')$ be any subrepresentation of $R$, and let $R''=(a'',b'',c'')$ be the corresponding quotient representation. Then, we have a diagram
\begin{equation*}
\xymatrix@C=3.5em{
0\ar[r]&\bullet\ar[r]^<{a'}\ar[d]&\bullet\ar[r]^<{b'}\ar[d]&\bullet\ar@{}[r]^<{c'}\ar[d]&\\
0\ar[r]&\bullet\ar[r]^<{1}\ar[d]&\bullet\ar[r]^(.3){2k+2}\ar[d]&\bullet\ar[r]^<{1}\ar[d]&0\\
&\bullet\ar[r]^(.17){a''}&\bullet\ar[r]^(.17){b''}&\bullet\ar[r]^(.17){c''}&0
}
\end{equation*}
(the fact that $R$ is $(\sigma,\gamma)$-globally injective and surjective implies that $R'$ is still injective, though not necessarily surjective, and that the quotient representation $R''$ preserves surjectivity).

$R$ is $\lambda$-semistable if we can find $\lambda=(\lambda_1,\lambda_2,\lambda_3)\in\ZZ^3$ such that
\[
\langle(\lambda_1,\lambda_2,\lambda_3),(1,2k+2,1)\rangle= \lambda_1+(2k+2)\lambda_2+\lambda_3=0
\]
and
\[
\langle(\lambda_1,\lambda_2,\lambda_3),(a',b',c')\rangle\geq 0.
\]
It is immediate from the diagram that either $a'=0$ or $a'=1$.

Suppose first that $a'=0$.
Then $b'$ must be either $2k+1$ or $2k+2$ since the cohomology bundle $E'$ of the monad that corresponds to $R'$ has $\rk(E')=b'-a'\geq \dim X=2k+1$.

When $a'=0$ and $b'=2k+1$, we see that $b''=1$, $c''=0$ and hence $c'=1$. So, $R$ is $\lambda$-semistable if
\[
\langle(\lambda_1,\lambda_2,\lambda_3),(0,2k+1,1)\rangle=(2k+1)\lambda_2+\lambda_3> 0.
\]

When $a'=0$ and $b'=2k+2$, we see again that $b''=c''=0$ and so $c'=1$. The $\lambda$-semistability of $R$ implies
\[
\langle(\lambda_1,\lambda_2,\lambda_3),(0,2k+2,1)\rangle=(2k+2)\lambda_2+\lambda_3> 0.
\]

Now suppose $a'=1$. In this case $b'=2k+2$, so that $b''=c''=0$ and thus $c'=1$, that is $R'=R$ and we must have
\[
\lambda_1+(2k+2)\lambda_2+\lambda_3=0.
\]

Hence, we can choose the triple $\lambda=(-1,0,1)$ satisfying all the required inequalities in order to $R$ to be $\lambda$-semistable.

The irreducibility statement follows from Theorem \ref{theorem2} and the general setting above described.
\end{proof}

The following is a consequence of Theorem \ref{theorem2} and Theorem \ref{modulispace}.

\begin{corollary}\label{irreduciblemodulispace}
Let $\mathcal{V}_{2k,1}$ be the set of rank $2k$ vector bundles which are the cohomology of a monad of type
\begin{equation*}
\xymatrix@R=\linhaxylabels@C=\colunaxylabels{0\ar[r]&\opn(-1)\ar[r]^-{f}&\OO_\pn^{\,2k+2}\ar[r]^-{g}&\opn(1)\ar[r]&0.}
\end{equation*}
Then the coarse moduli space $\mathcal{M}(\mathcal{V}_{2k,1})$ of $\lambda$-semistable vector bundles in $\mathcal{V}_{2k,1}$ is irreducible.
\end{corollary}

Naturally, irreducibility of the moduli space will be guaranteed in each case where we get an irreducible family, as mentioned in the general setting described after Theorem \ref{theorem2}.

\begin{Remark}
When $c=2$ an analogous study leads us to the conclusion that there is no $\lambda$ such that a representation $R=(2,2k+2,2)$ is $\lambda$-semistable. Therefore, in this case we are not able to construct the moduli space $\mathcal{M}(\mathcal{V}_{2k,2})$ with the help of King's construction.
\end{Remark}

\providecommand{\bysame}{\leavevmode\hbox to3em{\hrulefill}\thinspace}
\providecommand{\MR}{\relax\ifhmode\unskip\space\fi MR }
\providecommand{\MRhref}[2]{%
  \href{http://www.ams.org/mathscinet-getitem?mr=#1}{#2}
}
\providecommand{\href}[2]{#2}

\end{document}